\documentclass[12pt]{amsart}
\usepackage{amsthm,amsmath,amssymb,tikz-cd}
\usepackage{graphicx}

{\end{list}}
{
   \newtheorem{theorem}{Theorem}[section]
   \newtheorem{proposition}[theorem]{Proposition}
   \newtheorem{lemma}[theorem]{Lemma}

}
{\theoremstyle{definition}

}

\newcommand{\RR}{{\mathbb{R}}}

\newcommand{\CH}{\operatorname{CH}}
\newcommand{\res}{\operatorname{res}}

\newcommand{\Link}{\operatorname{Link}}
\newcommand{\Span}{\operatorname{Span}}

\newcommand{\Star}{\operatorname{Star}}

\newcommand{\MW}{Minkowski weight }

\newcommand{\setmin}{\,\protect%
\begin{picture}(8,10)\qbezier(1,5.5)(4,4.)(7,2.5)\end{picture}\,}

\renewcommand{\setminus}{\setmin}
\newcommand{\bs}{\smallsetminus}

\begin{document}
\title{Hodge Theory of Matroids by Star
Subdivisions}

\author{Andy Hsiao}
\address{Mathematics Department\\ University of British Columbia \\
  1984 Mathematics Road\\
Vancouver, B.C. Canada V6T 1Z2}
\email{ahsiao@math.ubc.ca}

\begin{abstract}
    It is known that the Chow ring of a matroid satisfies the Hard Lefschetz property and Hodge Riemann relations. We provide a proof of this theorem by decomposition of the deletion operation into star subdivisions at two dimensional cones. This gives a simple combinatorial proof which avoids intersection cohomology or algebraic geometry.
\end{abstract}

\maketitle

\section{Introduction}
Matroids are combinatorial objects which capture the essence of independence. Before defining a matroid, we first give the special case of representable matroids. Let $V$ be a vector space, and $E= \{v_1, \dots, v_n\} \subset V$ be a finite set of vectors. Then the representable matroid $M$ defined by $E$ is the set $E$, along with its lattice of flats The flats of $M$ are subsets of $E$ of the form $W \cap E$, where $W$ is a linear subspace of $V$.

A basic example is the Boolean matroid on $n$ elements, which is a realizable matroid on the set $E = \{e_1, \dots, e_n\}$, where $E$ is a basis of an $n$-dimensional vector space $V$. In this case, every subset of $E$ is a flat.

Generally, a matroid $M$ on a set $E = \{ 1, \dots, n+1\}$ is a specification a collection of subsets of $E$, called flats, such that
\begin{enumerate}
    \item $E$ is a flat.
    \item For any two flats $F,G$, their intersection $F \cap G$ is also a flat.
    \item For any flat $F$ and any element $x \in E \setminus F$, there is a unique flat $F'$ amongst the flats minimally containing $F$, such that $x \in F'$.
\end{enumerate}

We will further assume that the empty set is a flat of $M$. A flat is said to be nontrivial if it is nonempty and proper. The set of flats form a poset $L=L(M)$ with respect to inclusion. The rank of $M$ is defined as $\text{rank}(L)-1$. We refer to \cite{Oxley} for more on matroids.

In \cite{FY}, Feichtner and Yuzvinsky defined the Chow ring of a matroid. They showed that for any representable matroid, there is a complete complex projective variety with cohomology equal to the Chow ring of the matroid. This in turn shows that the Chow ring of the matroid satisfies the Hard Lefschetz property and Hodge Riemann relations. In \cite{AHK}, Adiprasito, Huh, and Katz extended this result to show that the Chow ring of all matroids satisfies the Hard Lefschetz property and Hodge Riemann relations through considering matroidal flips. Another proof of the theorem was given in \cite{BHM} by deletion of elements. In the case of representable matroids, this is deletion of a vector from $E$. The aim of this article is to provide a simple combinatorial proof of the Hard Lefschetz property and Hodge Riemann relations. We do so by decomposing the deletion operation into more elementary steps.

The following theorem is first proved in \cite{AHK}.
\begin{theorem}\label{maintheorem}
    Let $M$ be a matroid of rank $d+1$, $\ell \in \CH^1(M)$ a strictly convex class, and $0 \leq i \leq d/2$. Then multiplication by $\ell$ defines an isomorphism
    \begin{align*}
        \CH^i(M) &\to \CH^{d-i}(M)\\
        x &\mapsto \ell^{d-2i}x.
    \end{align*}

    Further, on $P^i =\ker(\CH^i(M) \xrightarrow{\ell^{d-2i+1}} \CH^{d-i+1}(M))$, the quadratic form $(-1)^i Q_\ell$ is positive definite, where $Q_\ell$ is
    \begin{align*}
        Q_\ell:  P^i &\to \RR\\
        x &\mapsto \deg(x^2 \ell^{d-2i}).
    \end{align*}
\end{theorem}
The Chow ring $\CH(M)$ of $M$ is defined in terms of the Stanley Reisner ring of the Bergman fan. The Bergman fan $\Delta_M$ of a matroid $M$ is a fan with cones defined by chains of nontrivial flats.

The original proof of the theorem in \cite{AHK} is by matroidal flips. In the case of a Boolean matroid, the Bergman fan is obtained by a barycentric subdivision of a complete fan. Each star subdivision in this sequence corresponds to a matroidal flip. For arbitrary matroids, the matroidal flips generalizes this process.

In \cite{BHM}, the authors give a proof of the theorem using the algebraic geometry of semismall maps from deletion of elements. Given any element $i \in E$, the deletion of $i$ from $M$ is the matroid $M \setminus i$ is the matroid on $E \setminus i$ with flats $F \setminus i$, where $F$ is a flat of $M$. We say an element $i \in E$ is a coloop of $M$ if after deleting $i$, the rank of $M \setminus i$ is lower than the rank of $M$. After deleting all non-coloops, the matroid becomes the Boolean matroid, whose fan is complete. Deleting a non-coloop $i$ gives rise to a map of fans $\pi: \Delta_{M} \to \Delta_{M \bs i}$, which was shown in \cite{BHM} to correspond to a semismall map of varieties. In \cite{dCM}, the authors showed that pull back of ample classes under semismall maps of varieties still satisfies Hard Lefschetz properties and Hodge Riemann relations. A combinatorial proof for toric varieties is given in \cite{KaruRelative}. This in turn shows $\Delta_M$ satisfies the Hard Lefschetz properties and Hodge Riemann relations if $\Delta_{M \bs i}$ does.

Our approach is to decompose the deletion map $\pi:\Delta_M \to \Delta_{M \bs i}$ into more elementary steps, as outlined in \cite{HKY}. There is a sequence of maps
\begin{align*}
    \Delta_{M} = \Delta_0 \to \Delta_1 \to \dots \to \Delta_k \to \Delta_{M \bs i}
\end{align*} which decomposes $\pi$, where each map $\Delta_j \to \Delta_{j+1}$ is a star subdivision at a two dimensional cone, and the last map $\Delta_k \to \Delta_{M \bs i}$ is a projection which induces an isomorphism of Chow ring. Given a strictly convex class $\ell \in \CH(\Delta_{M \bs i})$, the pullback $\pi^\ast \ell$ is a convex class in $\CH(\Delta_M)$. We show that the pullback of $\ell$ defines Hard Lefschetz properties and Hodge Riemann relations for all $\Delta_i$. Then by a deformation argument, we recover the proof of Theorem \ref{maintheorem} for all strictly convex classes on $\Delta_M$. This provides a combinatorial proof, which avoids intersection cohomology, and involves simpler combinatorics than the matroidal flips used in \cite{AHK}.

A similar techinque is used in \cite{ADH} to prove the conormal fan of $M$ satisfies the Hard Lefschetz property and Hodge Riemann relations. The authors define Lefschetz fans as those that satisfy Hard Lefschetz and Hodge Riemann, and further that the star of every cone is again a Lefschetz fan. They then showed that this property is preserved under star subdivisions and undoing star subdivisions.

\section{Poincar\'e duality algebras}
For this section, $H$ is a graded $\RR$-algebra with highest degree $d$ satisfying Poincar\'e duality. That is, there is a non-degenerate paring
    \begin{align*}
        H^i \times H^{d-i} \to H^d \cong \RR
    \end{align*} for all $i$, where the first pairing is given by multiplication in $H$, and we denote the isomorpism as $\deg\colon H^d \to \RR$. For $\ell \in H^1$, we say that the pair $(H,\ell)$ satisfies the \emph{Hard Lefschetz property} if for all $0 \leq i \leq d/2$, the linear map given by
        \begin{align*}
            H^i &\to H^{d-i}\\
            x &\mapsto \ell^{d-2i}x
        \end{align*} is an isomorphism. The bilinear form associated to $\ell$ is the map
        \begin{align*}
            B_\ell: H^{i} \times H^{i} \to \RR, (x,y) \mapsto \deg(x\cdot\ell^{d-2i}y).
        \end{align*}
        We then say $(H,\ell)$ satisfies the \emph{Hodge Riemann relations} if on the primitive part
        \[
            P_\ell^i = \ker(H^i \xrightarrow{\ell^{d-2i+1}} H^{d-i+1}),
        \] the quadratic form $(-1)^iQ_\ell$ is positive definite, where
        \begin{align*}
            Q_\ell \colon P_\ell^i \to \RR, x \mapsto B_\ell(x,x).
        \end{align*}

    For $\ell \in H^1$, the pair $(H,\ell)$ satisfies the Hard Lefschetz property if and only if the bilinear form $B_\ell$ is non-degenerate. Indeed, for any nonzero $y \in H^i$, by Poincar\'e duality, $\deg(x\cdot \ell^{d-2i}y)=0$ for all $x$ if and only if $\ell^{d-2i}y=0$.

    Assuming $(H,\ell)$ satisfies the Hard Lefschetz property, for each $m$, we get an isomorphism $\ell^{d-2m-2}: H^{m-1}\to H^{d-m+1}$. This then gives a decomposition $H^m = P^m \oplus \ell H^{m-1}$. Inductively, this gives the Lefschetz decomposition
    \begin{equation*}
        H^{m} = \bigoplus_{i=0}^m \ell^iP^{m-i}.
    \end{equation*} Further, this decomposition is orthogonal with respect to $B_\ell$. Indeed, for any $\ell^i f \in \ell^iP^{m-i}, \ell^j g \in \ell^jP^{m-j}$. Then $B_\ell(\ell^i f, \ell^j g) = \deg(\ell^{d-2m+i+j}fg)$. However, $\ell^{d-2m+i+j}g =0$ if $i > j$, since $d-2m+i+j > d-2(m-j)$.
    \begin{lemma}\label{HR-signature}
        Let $\ell \in H^1$ and suppose $(H,\ell)$ satisfies the Hard Lefschetz property. Then $(H,\ell)$ satisfies the Hodge Riemann relations if and only if the signature of the $Q_\ell$ on $H^m$ is
        \[
            \sum_{i=0}^m (-1)^i\dim(P_\ell^i).
        \]
    \end{lemma}
    \begin{proof}
        We show this by induction on $m$. For $m=0$, this is immediate as $H^0=P_\ell^0 \cong \RR$, and a form on $\RR$ is positive definite if and only if its signature is 1. For $m \geq 2$, we have the orthogonal decomposition
        \[H^m = P_\ell^m \oplus \ell H^{m-1}.\]
        So the signature of $Q_\ell$ on $H^m$ is the sum of its signature of $P_\ell^m$ and its signature on $H^{m-1}$. By induction hypothesis, the signature of $Q_\ell$ on $H^{m-1}$ is $\sum_{i=0}^{m-1}(-1)^i\dim(P_\ell^i)$. So $Q_\ell$ has signature $\sum_{i=0}^m (-1)^i\dim(P_\ell^i)$ on $H^m$ if and only if it has signature $(-1)^m\dim(P_\ell^m)$ on $P_\ell^m$. And the statement that $(-1)^iQ_\ell$ is positive definite on $P_\ell^m$ is exactly when it has signature $(-1)^m\dim(P_\ell^m)$ on $P_\ell^m$.
    \end{proof}

    In the case where $(H, \ell)$ satisfies the Hodge Riemann relations, a deformation of $\ell$ will also satisfy the Hodge Riemann relations, given every form in between also satisfies the Hard Lefschetz property.

    \begin{proposition}\label{HL-implies-HR}
        Let $U \subset H^1$ be a connected subset. Suppose that for all $\ell \in U$, the pair $(H,\ell)$ satisfies the Hard Lefschetz property. Then for all $\ell \in U$ and $0 \leq i \leq d/2$, the bilinear forms $Q_\ell$ has the same signature on $H^i$. In particular, if for some $\ell_0 \in U$, $(H, \ell_0)$ satisfies the Hodge Riemann relations, then for all $\ell \in U$, $(H,\ell)$ satisfies the Hodge Riemann relations.
    \end{proposition}
    \begin{proof}
        Since each pair $(H, \ell)$ satisfies the Hard Lefschetz property, each bilinear from $Q_\ell$ has full rank. Then for each $i$, the map $U \to \RR$ sending $\ell$ to the signature of $Q_\ell$ on $H^i$ is continuous. Now the map is further integral, so it must be constant on $U$.

        Suppose for some $\ell_0 \in U$, $(H, \ell_0)$ satisfies the Hodge Riemann relations. Since for each $\ell \in U$, $Q_\ell$ have the same signature as $Q_{\ell_0}$, by Lemma \ref{HR-signature}, each $\ell$ also satisfies the Hodge Riemann relations.
    \end{proof}

    We will also need that tensor products of rings satisfying Hard Lefschetz properties and Hodge Riemann relations also satisfies Hard Lefschetz and Hodge Riemann. A proof of this fact can be found in \cite{AHK}.

    \begin{lemma}\label{product-HL-HR}
        Let $H_1,H_2$ be graded rings satisfying Poincar\'e duality, and let $\ell_1 \in H_1^1$ and $\ell_2 \in H_2^1$. Suppose $(H_1, \ell_1)$ and $(H_2,\ell_2)$ both satisfies Hard Lefschetz and Hodge Riemann, then so does $(H_1 \otimes H_2, \ell_1 \otimes 1 + 1 \otimes \ell_2)$.
    \end{lemma}
\section{Fans}
Let $V$ be a $n$ dimensional vector space over $\RR$. A \emph{convex polyhedral cone} in $V$, or just cone, is a set
    \[
        \sigma = \left\{ \sum_{i=1}^m a_iv_i: a_i\in \RR_{\geq 0}\right\},
    \] where $v_1, \dots, v_m$ are vectors in $V$. A face of $\sigma$ is a subset of the form $\sigma \cap H_m$, where $H_m$ is the kernel of a linear map $m: V \to \RR$ such that $m(x) \geq 0$ for all $x \in \sigma$. A one dimensional face is called a ray. A cone is simplicial if $v_1, \dots, v_m$ are linearly independent.

    A \emph{fan} $\Delta$ is a collection of cones such that if $\sigma \in \Delta$ and $\tau$ a face of $\sigma$, then $\tau \in \Delta$. Further, if $\sigma_1, \sigma_2 \in \Delta$, then $\sigma_1 \cap \sigma_2$ is a common face of $\sigma_1$ and $\sigma_2$. For each ray $\rho_i \in \Delta$, we fix a generator $v_i$. A fan is simplicial if it consists of only simplicial cones. The combinatorial data of a simplicial fan is equivalent to a simplicial complex, although a fan further specifies an embedding in $V$. All fans in this paper will be simplicial.

    Let $\Delta$ be a fan in $V$. For a cone $\tau \in \Delta$, we define the star of $\tau$ in $\Delta$ as the subfan
    \begin{align*}
        \Star_\Delta \tau = \{ \sigma \in \Delta \mid \text{there exists } \gamma \in \Delta \text{ such that } \tau \leq \gamma \text{ and } \sigma \leq \gamma\}.
    \end{align*} We define the link of $\tau$ in $\Delta$ as the subfan
    \begin{align*}
        \Link_{\Delta} \tau = \{ \sigma \in \Star_\Delta \tau \mid \tau \cap \sigma = 0\}.
    \end{align*} We consider $\Star_\Delta \tau$ as embedded in $V$ as a subfan of $\Delta$. However, we consider $\Link_\Delta \tau$ as embedded in $V/ \Span{\tau}$.

    Let $\Delta_1$ and $\Delta_2$ in $V_1$ and $V_2$ be two fans. A morphism of fans $f: \Delta_1 \to \Delta_2$ is a linear map $\tilde{f}: V_1 \to V_2$ such that for any cone $\sigma_1$ of $\Delta_1$, we have $\tilde{f}(\sigma_1) \subset \sigma_2$ for some cone $\sigma_2$ of $\Delta_2$.

    For two fans $\Delta_1$ and $\Delta_2$ in $V_1$ and $V_2$ respectively, we define the product fan as
    \begin{align*}
        \Delta_1 \times \Delta_2 = \{ \sigma_1 \times \sigma_2 \subset V_1 \times V_2: \sigma_1 \in \Delta_1, \sigma_2 \in \Delta_2\}.
    \end{align*} A fan $\Delta$ in $V_1 \times V_2$ can have the same combinatorial data as the abstract fan $\Delta_1 \times \Delta_2$, but need not have the same embedding.

\subsection{Stanley Reisner rings}
    Let $\Delta$ be a simplicial fan in $V$ of pure dimension $d$ with ray generators $v_1, \dots, v_N$. The \emph{Stanley-Reisner ring} of $\Delta$ is
    \[
        \mathcal{A}(\Delta) = \RR[x_1, \dots, x_N]/I,
    \] where $I$ is generated by the set of square-free monomials $\prod_{i \in S}x_i$, where $\{v_i \} _{i \in S}$ does not generate a cone. This ring is graded by degree. The Stanley-Reisner ring $\mathcal{A}(\Delta)$ is also isomorphic to the ring of piecewise polynomial functions on $\Delta$, with the isomorphism given by mapping $x_i$ to a piecewise linear function such that $x_i(v_j) = \delta_{ij}$.

    Let $\theta_1, \dots, \theta_n$ be a basis of $V^\ast$, viewed as global linear functions on $V$. Identifying the Stanley Reisner ring with piecewise polynomial functions on $\Delta$, we have $\theta_1, \dots, \theta_n \in \mathcal{A}^1(\Delta)$. We then define the \emph{Chow ring} of $\Delta$ as
    \[
        \CH(\Delta) = \mathcal{A}(\Delta)/(\theta_1, \dots, \theta_n).
    \] The Chow ring $\CH(\Delta)$ is a graded ring, and is spanned by the monomials $x_\sigma = \prod_{\rho_i \in \sigma} x_i$, for all cones $\sigma \in \Delta$. In particular, $\CH(\Delta)$ has highest degree at most $\dim(\Delta)$.

    Let $f: \Delta_1 \to \Delta_2$ be a morphism of fans. There is then a map $f^\ast:\mathcal{A}(\Delta_2) \to \mathcal{A}(\Delta_1)$ defined by pullback of piecewise polynomial functions. Global linear functions pull back to global linear functions, so this further gives a morphism of Chow rings $\CH(\Delta_2) \to \CH(\Delta_1)$.

    Let $\Delta_1 \times \Delta_2$ be the product of two fans. On each cone $\sigma_1 \times \sigma_2$, the set of polynomials are given by $\mathcal{A}(\sigma_1 \times \sigma_2) = \mathcal{A}(\sigma_1) \otimes \mathcal{A}(\sigma_2)$. So as piecewise polynomials, we have the isomorphism of rings
    \[
        \mathcal{A}(\Delta_1 \times \Delta_2) \cong \mathcal{A}(\Delta_1) \otimes \mathcal{A}(\Delta_2).
    \]
    Taking the quotient by global linear functions, we get the isomorphism of rings
    \[
        \CH(\Delta_1 \times \Delta_2) \cong \CH(\Delta_1) \otimes \CH(\Delta_2).
    \]

\subsection{Minkowski weights}
For a $d$-dimensional fan $\Delta$ in $V$, a \emph{Minkowski weight of degree $d$} is a vector of real numbers $(w_\sigma)_{\sigma \in \Delta^d}$ satisfying the following balancing condition. For any $\tau \in \Delta^{d-1}$, let $\sigma_1, \dots, \sigma_s$ be the $d-$dimensional cones containing $\tau$. Denote by $v_i$ the ray in $\sigma_i \setminus \tau$, then
    \[
        \sum_{i=1}^s w_{\sigma_i}v_i \in \Span{\tau}.
    \] We denote the set of all Minkowski weights on $\Delta$ by $\MW(\Delta)$.

    Each Minkowski weight defines a linear map $\deg_w: \CH^d(\Delta) \to \RR$ given by $\deg_w(x_\sigma) = w_\sigma$, where $x_\sigma = \prod_{\rho_i \in \sigma} x_i$. This defines an isomorphism $\MW(\Delta) \cong \CH^d(\Delta)^\ast$. A proof of this isomorphism is given in \cite{HKY}.

    For a cone $\tau \in \Delta$, consider the projection map $\pi: V \to V/\Span{\tau}$. This defines a map of fans $\pi: \Star_\Delta\tau \to \Link_\Delta \tau$, which induces the pullback map of Chow rings
    \begin{align*}
        \pi^\ast: \CH(\Link_{\Delta}\tau) & \to \CH(\Star_\Delta \tau) \\
        x_\gamma &\mapsto x_\gamma.
    \end{align*}

    \begin{lemma}\label{Link-Star}
        The above map $\pi^\ast$ is an isomorphism of algebras.
    \end{lemma}
    \begin{proof}
        For any $f \in \mathcal{A}(\Star_\Delta \tau)$, we look for linear relations $\theta_i$ to eliminate variables $x_k$ of $f$, where $\rho_k \in \tau$. Then $f$ can be written by using only $x_k$ for $\rho_k \in \Link_\Delta \tau$, which allows us to construct an inverse to $\pi^\ast$.

        Let $v_1, \dots, v_m$ be the ray generators of $\tau$. It suffices for $1\leq j \leq m$ to write each $x_j$ as a linear combination of variables $\{x_k\}_{\rho_k \notin \tau}$. Now each $\theta_i$ can be viewed as a linear function on $V$, so there are linear combinations $\tilde{\theta_1}, \dots \tilde{\theta_m}$ with
        \begin{align*}
            \tilde{\theta_j}(v_i) = \begin{cases}
                1 & \text{if } i = j\\
                0 & \text{otherwise}.
            \end{cases}
        \end{align*} So in $\CH(\Delta)$, we have
        \[
            0 = \tilde{\theta_j} = x_j + \sum_{\rho_k \notin \tau} a_kx_k.
        \] Hence we can eliminate all $x_j$ for $\rho_j \in \tau$ by writing it as a linear combination of $x_k$ for $\rho_k \notin \tau$.
    \end{proof}

    From the inclusion of fans $\iota: \Star_\Delta \tau \hookrightarrow \Delta$, there is an induced map on Chow rings $ \iota^\ast: \CH(\Delta)\to \CH(\Star_\Delta \tau)$. We compose this with the map to links to get
    \begin{align*}
        (\pi^\ast)^{-1} \circ \iota^\ast : \CH(\Delta) &\to \CH(\Link_\Delta \tau)\\
        f &\mapsto f\vert_{\Link \tau}.
    \end{align*}

    For any Minkowski weight on $\Delta$, we can restrict it to a Minkowski weight on $\Link_\Delta\tau$ first by restricting onto $\Star_\Delta \tau$, then projecting. The map is given by
    \begin{align*}
        \res : \MW(\Delta) &\to \MW(\Link_{\Delta} \tau) \\
        (w_\sigma)_{\sigma \in \Delta} & \mapsto (w_\sigma)_{\sigma/\tau \in \Link_\Delta \tau}.
    \end{align*}

    We can also define a map $\CH(\Link_\Delta \tau) \to \CH(\Star_\Delta \tau) \to \CH(\Delta)$, where the first map is $\pi^\ast$, and the second map is multiplication by $x_\tau$. This map is given by $f \mapsto x_\tau \cdot \pi^\ast(f)$ for $f \in \CH(\Link_\Delta \tau)$.

    The above maps relate to each other in a simple adjunction formula. For any $w \in \MW(\Delta)$ and $f \in \CH(\Link_\Delta \tau)$, we have
    \begin{equation} \label{adjunction}
        \deg_w(x_\tau \cdot \pi^\ast(f)) = \deg_{\res w}(f).
    \end{equation} Indeed, for any maximal cone $\gamma \in \Link_\Delta \tau$, we have $\deg_w(x_\tau \cdot \pi^\ast(x_\gamma)) = w_{\tau + \gamma} = \deg_{\res(w)}(x_\gamma)$.

\subsection{Convexity}
Let $\Delta$ be a fan in $V$, and let $\ell \in \CH^1(\Delta)$. We say that $\ell$ is
    \begin{enumerate}
        \item \emph{positive} if $\tilde{\ell}(v_i) > 0 $ for some representative $\tilde{\ell}$ of $\ell$ and all $\rho_i \in \Delta$,
        \item \emph{non-negative} if $\tilde{\ell}(v_i) \geq 0 $ for some representative $\tilde{\ell}$ of $\ell$ and all $\rho_i \in \Delta$,
        \item \emph{strictly convex} if $\ell\vert_{\Link \tau}$ is positive for all $\tau \in \Delta$,
        \item \emph{convex} if $\ell\vert_{\Link \tau}$ is non-negative for all $\tau \in \Delta$.
    \end{enumerate}

The set of all strictly convex classes forms an open convex cone in $\CH^1(\Delta)$. We denote the set of all strictly convex classes by $K(\Delta)$. In the case where $K(\Delta)$ is nonempty, the set of convex classes is the closure of $K(\Delta)$.

If $f: \Delta_1 \to \Delta_2$ is a map of fans, then for any convex class $\ell \in \CH^1(\Delta_2)$, the pullback $f^\ast(\ell)$ is also convex. Non-negative representatives pull back as piecewise linear functions to non-negative piecewise linear functions. If $\tau$ maps to $\sigma$ under $f$, then $f$ defines a map $f_\tau: \Link_{\Delta_1} \tau \to \Link_{\Delta_2} \sigma$. Then for a non-negative representative $\tilde{\ell}$, $f^\ast(\tilde\ell)\vert_{\tau} = f_\tau^\ast(\tilde{\ell}\vert_{\sigma})$, which is non-negative.

In the case where $\CH(\Delta)$ satisfies Poincar\'e duality, we say $\CH(\Delta)$ satisfies the \emph{Hard Lefschetz property} if for all strictly convex $\ell \in K(\Delta)$, $(\CH(\Delta), \ell)$ satisfies the Hard Lefschetz property. Similarly, we say $\CH(\Delta)$ satisfies the \emph{Hodge Riemann relations} if for all strictly convex $\ell \in K(\Delta)$, $(\CH(\Delta), \ell)$ satisfies the Hodge Riemann relations.

\begin{proposition}\label{HR-implies-HL}
    Let $\Delta$ be a in $V$ of dimension $d$. Suppose that $\CH(\Delta)$ satisfies Poincar\'e duality, and for every ray $\rho \in \Delta^1$, the ring $\CH(\Link_\Delta \rho)$ satisfies the Hodge Riemann relations. Then $\CH(\Delta)$ satisfies the Hard Lefschetz property.
\end{proposition}
\begin{proof}
    By Poincar\'e duality, the spaces $\CH^i(\Delta)$ and $\CH^{d-i}(\Delta)$ are isomorphic, so it remains to prove that for a strictly convex class $\ell \in K(\Delta)$, the map given by multiplication with $\ell^{d-2i}$ is injective. Now suppose $\ell^{d-2i}\cdot g = 0$ for some $g \in \CH^{i}(\Delta)$, which in turn implies $\deg(g \cdot \ell^{d-2i}g) = 0$.

    Since $\ell$ is strictly convex, there is a representative such that
    \begin{align*}
        \ell = \sum_{\rho_i \in \Delta} a_ix_i, \qquad a_i > 0.
    \end{align*} Then by the adjunction (\ref{adjunction}), we have
    \begin{align*}
        0=\deg(\ell^{d-2i}\cdot g^2) &= \sum_{\rho_i \in \Delta} a_i\deg(x_i\ell^{d-2i-1}g^2)\\
        &= \sum_{\rho_i \in \Delta} a_i\deg_{\Link\rho_i}((\ell^{d-2i-1}g^2)\vert_{\Link \rho_i}).
    \end{align*}

    Since for any ray $\rho \in \Delta$, $\ell^{d-2i}g\vert_{\Link \rho} = 0$, $g\vert_{\Link \rho_i}$ is primitive. Now by assumption, every link satisfies Hodge Riemann relations, so we have $(-1)^i\deg_{\Link\rho_i}(\ell_{\rho_i}^{d-2i-1}g^2) \geq 0$, and is equal to 0 if and only if $g\vert_{\Link\rho_i}=0$. Since each $a_i>0$, this shows that $g \vert_{\Link_{\rho_i}} = 0$ for every ray $\rho_j$. Then $g \cdot x_i = 0$ for all i, and hence $g=0$.
\end{proof}

\section{Star subdivisions}
For this section, let $\Delta$ be a fan in $V$, and let $\tau$ be a fixed two dimensional cone of $\Delta$ generated by $v_1,v_2$. Let $s: \hat{\Delta} \to \Delta$ be a star subdivision at $\tau$ by introducing a new ray $v_0 = v_1+v_2$. A Minkowski weight $w$ on $\Delta$ defines a Minkowski weight $\hat{w}$ on $\hat\Delta$ by $\hat{w}_\sigma = w_{s(\sigma)}$. We fix a weight $w \in \MW(\Delta)$, and denote the degree maps induced by $w, \hat w, \res(w)$ by $\deg_{\Delta}, \deg_{\hat\Delta}, \deg_{\Link_\delta \tau}$ respectively.

The Chow ring of a star subdivision can be decomposed orthogonally in terms of Chow ring of the original fan. We refer to \cite{HKY} for a proof of the following lemma.

    \begin{lemma}\label{ortho-decomp}
        Let $s: \hat{\Delta} \to  \Delta$ be a star subdivision at a 2-dimensional cone $\tau$. There is a decomposition
        \begin{align*}
            \CH(\hat{\Delta}) = \CH(\Delta)\oplus x_0\CH(\Link_\Delta\tau),
        \end{align*} which is orthogonal with respect to the Poincar\'e paring $(x,y) \mapsto \deg(x,y)$.
    \end{lemma}

    Note that a class $\ell \in \CH^1(\Delta)$ acts on each summand. If $\ell$ satisfies the Hard Lefschetz property and Hodge Riemann relations on each summand, then it satisfies Hard Lefschetz and Hodge Riemann on the sum.

    \begin{proposition}\label{pullbackHLHR}
        Let $\ell \in \CH^1(\Delta)$. Then
        \begin{enumerate}
            \item If $(\CH(\Delta), \ell)$  and $(\CH(\Link_\Delta \tau), \ell\vert_{\Link \tau})$ satisfy the Hard Lefschetz property, then so does $(\CH(\hat{\Delta}), s^\ast\ell)$.
            \item If $(\CH(\Delta), \ell)$ and $(\CH(\Link_\Delta \tau), \ell\vert_{\Link \tau})$ satisfy the Hodge Riemann relations, then so does $(\CH(\hat{\Delta}), s^\ast\ell)$.
        \end{enumerate}
    \end{proposition}

    \begin{proof}
         From Lemma \ref{ortho-decomp}, we have an orthogonal decomposition
            \begin{align*}
                \CH(\hat{\Delta}) = \CH(\Delta)\oplus x_0\CH(\Link_\Delta\tau).
            \end{align*} Now for any $0 \leq i \leq d/2$, multiplication by $\ell^{d-2i}$ defines isomorphisms on the summands
            \begin{align*}
                \ell^{d-2i}&: \CH^i(\Delta) \to \CH^{d-2i}(\Delta),\\
                \ell\vert_{\Link \tau}^{d-2i}&: \CH^{i-1}(\Link_\Delta\tau) \to \CH^{d-2i-1}(\Link_\Delta\tau)
            \end{align*} by the assumption that they satisfies Hard Lefschetz. So multiplication by $\ell^{d-2i}$ defines an isomorphism $\CH^i(\hat\Delta) \xrightarrow{\sim} \CH^{d-2i}(\hat\Delta)$. This proves (1), and in particular also proves $\CH(\hat{\Delta})$ satisfies Poincar\'e duality.

        The kernel of the multiplication by $\ell^{d-2i+1}$ map decomposes as a direct sum
            \[
                P_\ell^i(\hat\Delta) = P_\ell^i(\Delta) \oplus x_0P_{\ell\vert_{\Link \tau}}^{i-1}(\Link_\Delta\tau),
            \] orthogonal with respec to $B_\ell$. By assumption $(-1)^iQ_\ell$ is positive definite on the first summand, we prove the same for the second summand.

        Now let $f \in \CH^{d-2}(\Link_\Delta \tau)$. Note that we can find a global linear function of the form $\tilde\theta= x_1 + x_0 + s^\ast\pi^\ast g$ on $\Star_\Delta \tau$ for some $g \in \CH^1(\Link_\Delta \tau)$. Then we have
            \begin{align*}
                \deg_{\hat\Delta}(x_0^2\cdot s^\ast \pi^\ast f) &= \deg_{\hat\Delta}(x_0(-x_1 - s^\ast\pi^\ast g)s^\ast\pi^\ast f) \\
                &= -\deg_{\hat\Delta}(x_0x_1 s^\ast\pi^\ast f) - \deg_{\hat\Delta}(x_0 s^\ast \pi^\ast (gf)).
            \end{align*} Since $\CH(\Link_\Delta \tau)$ has top degree $d-2$, $gf = 0$. Let $\tau'$ be the cone formed by the rays $\rho_0,\rho_1$. Note that $\Link_{\Delta}\tau = \Link_\Delta \tau'$. So by the adjunction (\ref{adjunction}), we have
            \[
              \deg_{\hat\Delta}(x_0^2 \cdot s^\ast\pi^\ast f) = -\deg_{\hat\Delta}(x_0x_1s^\ast\pi^\ast f) = - \deg_{\Link_\Delta \tau}(f).
            \] So the action of $Q_\ell$ on the second summand is
            \begin{align*}
                (-1)^i Q_\ell(x_0f) &= (-1)^i\deg_{\hat\Delta}(\ell^{d-2i}x_0^2 f^2)\\ & = (-1)^{i-1}\deg_{\Link_\Delta\tau}(f^2\ell\vert_{\Link \tau}^{d-2i}) = (-1)^{i-1}Q_{\ell\vert_{\Link \tau}}(f).
            \end{align*} By assumption $ (-1)^{i-1}Q_{\ell\vert_{\Link \tau}}$ is positive definite on the $\Link_\Delta\tau$, so $(-1)^iQ_\ell$ is positive definite on the second summand. This concludes the proof for (2).
    \end{proof}

\section{Matroids}

Let $M$ be a matroid of rank $d+1$ on the set $E = \{1,2, \dots, n+1\}$. Let $e_1, \dots, e_{n+1}$ be the standard basis of $\RR^E$. Recall that a matroid is defined by its lattice of flats $L$. For any subset $S \subset E$, let $e_S = \sum_{i \in S}e_i$. The Bergman fan $\Delta_M$ of $M$ is a simplicial fan embedded in $\RR^E/e_E$ defined as follows. For each flat $F$ of $M$, define a ray generated by the image of $e_F$ in the quotient. For each flag of nontrivial flats
\[
    \mathcal{F}: \emptyset \neq F_1 < F_2 < \dots < F_s \neq E,
\] there is a cone $\sigma_{\mathcal{F}} = \text{cone}\{e_{F_1}, \dots, e_{F_m}\} \in \Delta_M$. We further define the Chow ring $\CH(M)$ of $M$ as the Chow ring $\CH(\Delta_M)$ of the Bergman fan.

\subsection{Intervals}

Let $F,G$ be two flats of $M$. The interval $[F,G]$ is the sublattice
\[
    [F,G] = \{H \in L: F \leq H \leq G\}.
\] Each such lattice defines a matroid on $G \setminus F$, where for a flat $H$ of $M$ such that $F \leq H \leq G$, there is a flat $H \setminus F$. For simplicity of notation, we denote the induced matroid also as $[F,G]$. Let $F$ be a nontrivial flat of $M$. We define the localization of $M$ at $F$ as the matroid $ M^F = [\emptyset, F]$ and the contraction of $M$ by $F$ as the matroid $M_F = [F,E]$. We can alternatively describe the matroid of an interval as $[F,G] = (M_F)^G$.

\begin{lemma}\label{Bergman-link-product}
    Let $\sigma$ be a cone of the Bergman fan $\Delta_M$. Then $\Link_{\Delta_M}\sigma$ is a product of Bergman fans.
\end{lemma}
\begin{proof}
    We prove first for the case where $\sigma$ is a ray. Let $F$ be a flat of $M$, and $\rho_F$ the corresponding ray. Then a cone $ \tau \in \Link_{\Delta_M} \rho_F$ is obtained by taking a cone $\overline{\tau}$ in $\Delta_M$ that contains $\rho_F$, then removing $\rho_F$ to get $\tau = \overline\tau \setminus \rho_F$. In terms of chains, a cone $\tau$ in $\Link_{\Delta_M} \rho_F$ is obtained by taking a chain of flats which includes $F$, then removing $F$.

    Each such chain is defined by a chain in $[\emptyset, F]$ and a chain in $[F,E]$. This shows, as abstract fans, we have
    \[
        \Link_{\Delta_M}\rho = \Delta_{[\emptyset, F]} \times \Delta_{[F,E]}.
    \]

    We now show the link is also embedded with the product embedding. Recall that $\Link_{\Delta_M}\rho$ is embedded in $V/\Span(\rho)$. In the matorid case, the ambient space is $\RR^E/(e_E)/\Span(\rho_F) = \RR^E/(e_E,e_F) \cong (\RR^F/e_F) \times (\RR^{E\bs F}/e_{E\bs F})$. In our case, the product fan $\Delta_{[\emptyset, F]} \times \Delta_{[F,E]}$ is embedded in the same space. Let $\mathcal{F}$ be the chain of flats of a cone in $\Link_{\Delta_M} \rho$. Let $\mathcal{F}_-$ be the subchain of $\mathcal{F}$ consisting of flats contained in $F$, and let $\mathcal{F}_+$ be the subchain of flats containing $F$. Then $\mathcal{F}_-$ is embedded in $(\RR^F/e_F)$, and $\mathcal{F}_+$ is embedded in $(\RR^{E\bs F}/e_{E\bs F})$.

    For a general cone $\sigma$, choose any ray $\rho$. Define $\sigma'$ to be the subcone of $\sigma$ which includes all rays of $\sigma$ except $\rho$. Then defining $L = \Link_{\Delta_M} \sigma$, the conclusion follows inductively from the identity
    \[
        \Link_{\Delta_M} \sigma = \Link_{L}\rho.
    \]
\end{proof}

\subsection{Deletion}
For any element $i \in E$, define the deletion of $i$ from $M$ as the matroid $M \setminus i$ on $E = \setminus i$ where the flats of $M \setminus i $ are $F \setminus \{i\}$ for flats $F$ of $M$. For any $i$, the fan $\Delta_{M \bs i}$ is the image of $\Delta_M$ under the projection $\pi: \RR^E/e_E \to \RR^E/(e_E, e_i)$.

An element $i \in E$ is a coloop is $M \setminus i$ has lower rank than $M$. The projection map can be decomposed further as
\[
    \pi:\Delta_M = \Delta_0 \xrightarrow{\pi_0} \Delta_1 \xrightarrow{\pi_1} \dots \to \Delta_k \xrightarrow{\pi_k} \Delta_{M \bs i}.
\] Here, for each $j = 0, \dots, k-1$, the map $\pi_j: \Delta_j \to \Delta_{j+1}$ is a star subdivision at a two dimensional cone. The last map $\pi_k: \Delta_k \to\Delta_{M \bs i}$ is a projection map. We outline here the decomposition, but may omit proofs at places. For a more detailed explanation we refer to \cite{HKY}.

Let $S_i = \{F_1, \dots, F_k\}$ be the set of all nontrivial subsets $F \subset E$ such that $F$ and $F \cup i$ are both flats of $M$, with the ordering chosen to refine the partial order of the lattice. For $1 \leq j \leq k$, define the equivalence relation $\sim_j$ by $F_l \sim_j F_l \cup i$ for $l \leq j$, with no relation on the other elements. The poset $L/ \sim_j$ defines a fan $\Delta_j$, where the rays of $\Delta_j$ are the nontrivial equivalence classes in $L/\sim_j$, and cones are images of chains of nontrivial flats in $L$. The embedding is given by mapping the equivalence classes $\{F_l, F_l \cup i\}$ to the ray generated by $e_{F_l}$ for $l \leq j$, and mapping other flats $F$ to the ray generated by $e_F$ as before. For convenience, we may denote an equivalence class by a representative.

The map $\Delta_{j-1} \to \Delta_j$ is the star subdivision at the cone $\tau: i < F_j \cup i$. Before the subdivision, in $\Delta_j$, $F_j$ and $F_j \cup i$ corresponds to the same ray. After the subdivision, in  $\Delta_{j-1}$, the ray $\rho_{F_j \cup i}$ is the new ray that subdivides the cone $i \leq F_j$.

The last map $\Delta_k \to \Delta_{M \bs i}$ is defined by projecting $V \to V/\Span(e_i)$. The fan $\Delta_k$ has exactly one more ray than $\Delta_{M \bs i}$, namely the ray generated by $e_i$, which is contracted to $0$ under $\pi_k$. The projection is bijective on the set of remaining rays. The map $\pi_k$ induces a map of Chow rings $\CH(M \setminus i) \to \CH(\Delta_k)$. This map is futher an isomorphism.

\begin{lemma}\label{subdivision-link-product}
    Let $\tau$ be the cone $i < F_j \cup i$ in $\Delta_m$ for some $m =0, \dots, k$. Then $\Link_{\Delta_m} \tau = \Delta_{[i, F_j \cup i]} \times \Delta_{[F_j \cup i, E]}$ is a product of Bergman fans. In particular, $\Link_{\Delta_m} \tau$ is the same fan for all $m$.
\end{lemma}
\begin{proof}
    The cones in $\Link_{\Delta_m} \tau$ are images of chains that includes $i$ and $F_j \cup i$, but with the two cones removed. So $\Link_{\Delta_m} \tau$ is the image of $\Delta_{[i,F_j\cup i]} \times \Delta_{[F_j \cup i, E]}$ under the quotient. However the quotient is injective when restricted to the above product, as no such fans contain both $F_l$ and $F_l \cup i$ for any $l \leq k$. This shows that $\Link_{\Delta_m} \tau = \Link_{\Delta_M} \tau$, which by Lemma \ref{Bergman-link-product} is a product of Bergman fans.

        The rays of $\Delta_m$ are embedded in the same way as the rays of $\Delta_M$, except for $\rho_{F_l \cup i}$ for $1 \leq l \leq m$, which are embedded as the ray generated by $e_{F_l}$. However, both of $\Link_{\Delta_m} \tau$ and $\Link_{\Delta_M} \tau$ are embedded in $\RR^E/(e_E, e_i, e_F)$. So each $F_l \cup i$ is embedded as the ray generated by $e_{F_l} + e_i = e_{F_l}$ for both fans. This shows $\Link_{\Delta_m} \tau = \Link_{\Delta_M} \tau$ as embedded fans.
\end{proof}

We need also a lemma to ensure pullback along each map preserves convexity properties.

\begin{lemma}\label{convexity}
    Let $\ell \in \CH^1(M \setminus i)$ be a strictly convex class. Denote the pullbacks of $\ell$ as $\ell_m = \pi_m^\ast \circ \pi_{j+1}^\ast \circ \dots \circ \pi_k^\ast(\ell)$ for each $m = 0, \dots, k$. Then $\ell_m$ is convex on $\Delta_m$, and strictly convex on $\Link_{\Delta_m} (i < F_j \cup i)$ for all $0 \leq m,j \leq k$.
\end{lemma}
\begin{proof}

The class $\ell_m$ is the pullback of a convex class along the map of fans, so it is also convex.

By Lemma \ref{subdivision-link-product}, $\Link_{\Delta_m} (i < F_j \cup i) = \Link_{\Delta_k} (i < F_j \cup i)$ for every $m$. It suffices to prove the lemma for the case where $m = k$. The projection $\pi_k$ is surjective on the set of cones. So restricting to the link, the map $\pi_k\vert_{\Link}: \Link_{\Delta_k} (i < F_j \cup i) \to \Link_{\Delta_{M \bs i}} F_j$ defines an inclusion of fans. Furthermore, since both fans are embedded in $\RR^E/(e_E, e_i, e_{F_j})$, this is in fact an embedding of fans. The restriction of $\ell_j$ to $\Link_{\Delta_M} (i <  F_j\cup i)$ is the restriction of $\ell$ to $\Link_{\Delta_M} (i <  F_j\cup i)$ as a subfan of $\Link_{\Delta_{M \bs i}} F_j$, and is thus also strictly convex.

\end{proof}

\subsection{Proof of the main theorem}
We are now ready to prove Theorem \ref{maintheorem}, which states that the Chow ring $\CH(\Delta_M)$ satisfies the Hard Lefschetz property and Hodge Riemann relations.

\textit{Proof of Theorem \ref{maintheorem}.} We proceed by induction on the rank of $M$. The case where $\text{rank}(M) =0$ is trivial. Assume now the theorem holds for all matroids of rank $d$ or lower. Let $M$ be a matroid of rank $d+1$ on $E = \{1, \dots, n+1\}$.

Recall that $M$ is Boolean if every subset of $E$ is a flat. If $M$ has no non-coloops, then $M$ is Boolean. In this case, $\Delta_M$ is complete, and is the fan of a smooth projective toric variety. The Chow ring is then the cohomology of the toric variety $\CH(M) = \CH(X_{\Delta_M})$, which satisfies Hard Lefschetz and Hodge Riemann. Alternatively, to avoid algebraic geometry, we refer to \cite{FK} for a combinatorial proof of the Hard Lefschetz theorem and Hodge Riemann relations for complete simplicial fans.

In the case where $M$ has non-coloops, deleting every non-coloop gives the Boolean matroid of rank $d+1$. It then suffices to show that for any non-coloop $i \in M$, if $\CH(M\setminus i)$ satisfies the Hard Lefschetz properties and Hodge Riemann relations, then so does $\CH(M)$.

Suppose now $M$ has a non-coloop $i \in E$, and suppose $\CH(M \setminus i)$ satisfies the Hard Lefschetz properties and Hodge Riemann relations. Consider as above the sequence of maps
\[
    \Delta_M = \Delta_0 \xrightarrow{\pi_0} \Delta_1 \xrightarrow{\pi_1} \dots \to \Delta_k \xrightarrow{\pi_k} \Delta_{M \bs i}.
\]

Let $\ell \in \CH^1(M \setminus i)$ be a strictly convex class. By Lemma \ref{subdivision-link-product}, for each fan $\Delta_j$ and $\tau: i < F_j \cup i$, the link is $\Link_{\Delta_j} \tau = \Delta_{[i, F_j \cup i]} \times \Delta_{[F_j \cup i, E]}$. By induction hypothesis, both of $\Delta_{[i, F_j \cup i]}$ and $\Delta_{[F_j \cup i, E]}$ satisfies Hard Lefschetz and Hodge Riemann. By Lemma \ref{convexity} $\ell_j \vert_{\Link \tau}$ is strictly convex. Then $\ell_j \vert_{\Link \tau} = \ell' \otimes 1 + 1 \otimes \ell''$ for strictly convex classes $\ell' \in \CH^1(\Delta_{[i, F_j \cup i]})$, $\ell'' \in \CH^1(\Delta_{[F_j \cup i, E]})$. Then by Lemma \ref{product-HL-HR}, $(\Link_{\Delta_j} \tau, \ell_j \vert_{\Link \tau})$ satisfies Hard Lefschetz and Hodge Riemann. By Proposition \ref{pullbackHLHR}, $(\CH(\Delta_j), \ell_j)$ satisfies Hard Lefschetz and Hodge Riemann for every $j$. In particular, $(\CH(\Delta_M), \ell_0)$ satisfies the Hodge Riemann relations.

For every ray $\rho \in \Delta_M$ defined by a flat $F$, by Lemma \ref{Bergman-link-product}, $\Link_{\Delta_M} \rho = \Delta_{M_F} \times \Delta_{M^F}$. Again by induction hypotheses, $\Delta_{M_F}$ and $\Delta_{M^F}$ both satisfies Hard Lefschetz and  Hodge Riemann, and hence so does $\Link_{\Delta_M} \rho$. Since the links of each ray has rank $d$ or lower, by induction hypothesis and Proposition \ref{HR-implies-HL}, $\CH(M)$ satisfies the Hard Lefschetz property for all strictly convex classes in $\CH^1(M)$. Now the convex class $\ell_0 \in \partial K(\Delta_M)$ lies on the boundary of the cone of convex classes. Since $\ell_0$ satisfies Hodge Riemann relations, by Proposition \ref{HL-implies-HR}, we conclude that $\CH(M)$ also satisfies the Hodge Riemann relations. \hfill $\square$

\bibliographystyle{plain}
\bibliography{bib}{}

\end{document}